\documentclass{amsart}


\title[Estimates of the Derivative of the Entropy of Gaussian Thermostats]
      {Estimates of the Derivative of the Entropy of Gaussian Thermostats}

\author[A. Arbieto]{A. Arbieto}
\address{Instituto de Matem\'atica, Universidade Federal do Rio de Janeiro, P. O. Box 68530, 21945-970 Rio de Janeiro, Brazil.}
\email{arbieto@im.ufrj.br}

\author[A. Lopes]{A. Lopes}
\address{Instituto de Matem\'atica, Universidade Federal de Uberl\^andia, Minas Gerais, Brazil.}
\email{adilson@famat.ufu.br}

\thanks{Partially supported by CNPq, FAPERJ and PRONEX/DS from Brazil. The second author was supported by the grant ``Bolsa Nota 10'' by Faperj in his PhD. The first author wants to thanks the hospitality of the Federal University of Uberl\^andia.}

\newtheorem{theorem}{Theorem}
\newtheorem{corollary}[theorem]{Corollary}

\newtheorem{lemma}[theorem]{Lemma}

\theoremstyle{definition}

\newcommand{\R}{\mathbb{R}}
\newcommand{\eps}{\varepsilon}

\newcommand{\re}{\rightarrow}
\newcommand{\al}{\alpha}

\newcommand{\rea}{{\mathbb R}}

\newcommand{\ep}{\epsilon}
\newcommand{\om}{\omega}

\newcommand{\be}{\beta}
\newcommand{\ga}{\gamma}

\newcommand{\la}{\lambda}

\newcommand{\pt}{\varphi}
\newcommand{\ptt}{{\varphi}_{t}}

\newcommand{\te}{\theta}
\newcommand{\noi}{\noindent}

\newcommand{\ms}{\Sigma}
\newcommand{\vp}{{\varphi}_{t}}

\newcommand{\ld}{\lambda\mapsto}

\newcommand{\yv}{E-\frac{\left\langle E,v\right\rangle}{{\left|v\right|}^{2}}v}
\newcommand{\yg}{E-\frac{\left\langle E,\dot{\gamma}\right\rangle}{{\left|\dot{\gamma}\right|}^{2}}\dot{\gamma}}

\newcommand{\fg}{\kappa}
\newcommand{\ty}{\tilde{Y}}

\newcommand{\fl}{\varphi_{t}^{\lambda}}

\newcommand{\jx}{J_{\xi}}
\newcommand{\dej}{\dot{J_{\xi}}}
\newcommand{\ddj}{\ddot{J_{\xi}}}
\newcommand{\dg}{\dot{\gamma}}
\newcommand{\fv}{\frac{1}{{\left|v\right|}^{2}}}

\newcommand{\yvl}{E_{\lambda}(p)-\frac{\left\langle E_{\lambda}(p),v\right\rangle}{{\left|v\right|}^{2}}v}

\newcommand{\pe}{E'_{0}}
\newcommand{\pas}{\frac{\partial f}{\partial s}}
\newcommand{\pat}{\frac{\partial f}{\partial t}}

\newcommand{\gz}{\dot{\gamma}_{0}}
\newcommand{\gl}{\dot{\gamma}_{\lambda}}

\newcommand{\ml}{\Lambda}
\newcommand{\izt}{\int_{0}^{T}}
\newcommand{\ygpz}{\frac{\left\langle E'_{0},\dot{\gamma_{0}}\right\rangle}{{\left|\dot{\gamma_{0}}\right|}^{2}}\dot{\gamma}_{0}}
\newcommand{\frgz}{\frac{1}{{\left|\gamma_{0}\right|}^{2}}}

\newcommand{\yygl}{E_{\lambda}-\frac{\left\langle E_{\lambda},\dot{\gamma_{\lambda}}\right\rangle}{{\left|\dot{\gamma_{\lambda}}\right|}^{2}}\dot{\gamma}_{\lambda}}

\newcommand{\glb}{\bar{\gamma}_{\lambda}}
\newcommand{\glbp}{\dot{\bar{\gamma}}_{\lambda}}
\newcommand{\dct}{\frac{D}{dt}}
\newcommand{\dcl}{\frac{D}{d\lambda}}
\newcommand{\mt}{\Theta}
\newcommand{\tl}{T_{\lambda}}
\newcommand{\all}{{\alpha}^{\lambda}}
\newcommand{\bl}{{\beta}^{\lambda}}
\newcommand{\ene}{{\mathcal{E}}}

\begin{document}

\begin{abstract}
We consider a variation of an Anosov geodesic flow by Gaussian Thermostats and we obtain estimates of the derivative of the entropy map at the geodesic flow. In particular, we prove that the entropy of the geodesic flow is a local maximum for the entropy map.
\end{abstract}

\maketitle

\section{Introduction}

The topological entropy of a flow is well known quantity that measures the complexity of the flow. It is a central theme in the theory of dynamical systems and a lot of research was done to understand and calculate it. So given a flow with certain properties it is a natural question to obtain bounds for the entropy. Moreover, given a family of flows it is nice to know how the topological entropy varies in it. 

In this article, we deal with such question in a family formed by Gaussian Thermostats. These flows provide interesting models in nonequilibrium statistical mechanics, see \cite{Galavotti} and \cite{Ruelle}.

Let $(M,g)$ be a closed Riemannian manifold. We denote a point  in $TM$ by $(p,v)=\theta\in TM$, where $p\in M$ and $v\in T_pM$ . In these coordinates, we can write the geodesic flow (see \cite{P.livro}) as the following equation:
\begin{align}
	\frac{dp}{dt}=v\ \ \ \ \ \  \frac{Dv}{dt}=0.
	\label{equ.geo}
\end{align}

Let $t\re\left(\gamma(t),\gamma'(t)\right)$ be a curve in $TM$. This curve is an integral curve to the geodesic field $G(\theta)=(v,0)$ if and  only if
\begin{align}
	\frac{D}{dt}\gamma(t)=0.
\end{align}
This shows that the solutions to the equation in ($\ref{equ.geo}$) are exactly the same of the geodesic field $G$. Analogously, we can define  the Gaussian Thermostat vector field as:
\begin{align}
	\frac{dp}{dt}=v\ \ \ \ \ \ \ \ \ \  \frac{Dv}{dt}=\yv
	\label{d2tg}
\end{align}

Now, we assume that the geodesic flow $\vp:SM\re SM$ is Anosov. Let $\ld E(\la)$ be a family of vector fields on $M$ such that $E(0)=0$.

So we can consider the vector fields $F_{\la}$ given by  $F_{\la}(\te)=(v,\yvl)$. Let us denote by $\fl:SM\re SM$ the flow generated by $F_{\la}$.

We recall the concept of entropy (see \cite{Walters}). If $f:(X,d)\to (X,d)$ is a homeomorphism then let us define $d_n(x,y)=\max \{d(f^i(x),f^i(y));\textrm{ for }$ $i=0,\dots,n-1\}$.

A subset $A$ of $X$ is said to be $(n,\eps)$-separated if each pair of distinct points of $A$ is at least $\eps$ apart in the metric $d_n$. Denote by $N(n, \eps)$ the maximum cardinality of an $(n, \eps)$-separated set. The topological entropy of $f$ is defined by:
$$h_{top}(f)=\lim_{\eps\to 0}\limsup_{n\to \infty}\frac{1}{n}\log N(n,\eps).$$

If $X_t$ is a flow then we set $h_{top}(X_t)=h_{top}(X_1)$.
Let us define $h(\lambda):=h_{top}(\fl)$, the entropy map associated to this variation.

We remark that by structural stability the flows $\fl:SM\re SM$ are Anosov for $\la\in (-\eps,\eps)$, where $\eps$ is small.

We set $E'_0=\frac{dE_{\la}}{d\la}\big|_{\la=0}$. Now, we define the operator $Z:TM\re TM$ such that for every $(p,v)=\te\in TM$ we have
\begin{align*}
	Z_{\te}(\cdot)=\fv\left(\left\langle v,\cdot\right\rangle  \pe(p)    -   \left\langle \pe(p),\cdot\right\rangle v        \right).
\end{align*}
Our main result is the following estimate

\begin{theorem} The first derivative of the entropy map is zero, i.e., $h'(0)=0$. Moreover, if $Z\neq 0$ then
\begin{align*}
	h''(0)   \leq    -h(0)\frac{\left[\int_{SM}\left( \left|\pe(p)\right|^{2}  - \frac{\left\langle \pe(p),v\right\rangle^{2}}{\left|v\right|^{2}}                 \right)    dm              \right]^{2}}{  \int_{SM} \left( \left|\nabla_{v}\pe(p)\right|^{2} -\frac{\left\langle \nabla_{v}\pe(p),v\right\rangle^{2}}{\left|v\right|^{2}}  - \left\langle R(v,\pe(p))v,\pe(p)\right\rangle                     \right)dm                 } < 0
\end{align*}

\noi where $m$ is the maximal entropy measure of the geodesic flow  $\varphi_t^0$.
\label{teo1.1tg}
\end{theorem}

This immediately implies the following result.

\begin{corollary}
If $Z\neq 0$ then $\ld h(\la)$ has a strict local maximum at $\la=0.$
\end{corollary}

A particular case of the previous theorem is $E(\la)=\la E$ where $E$ is a fixed vector field. In this case $\pe=E$ and $Z=\ty$ (defined below). The motivation of these results comes from the articles \cite{PP} and \cite{P.deriv.geo}.  The proof also relies on Pollicott's formula \cite{Pollicott}. However, in the case of the Gaussian Thermostat flow, the computations are a little bit different. For instance,  the variational field produced by the flow has much more terms, and now they need to be estimate. Also, in the magnetic case we could  consider the magnetic field (the Lorentz force) as a $(1,1)-tensor$ and we can take covariant derivate to differentiate it. However, for  Gaussian Thermostats this do not  occurs. Indeed,  the  Gaussian Thermostat field (defined below) is not a tensor in general and this also complicates the computations.

\section{Preliminars}

As in the introduction, we set $(M,g)$  as a closed Riemannian manifold. Let $H:TM\re \rea$ be the energy functional $H(p,v)=\frac{1}{2}g_{p}(v,v)$ and $E\in\mathfrak{X}(M)$  be a vector field on $M$.  

As we said before, the Gaussian Thermostat can be defined as the flow on $TM$ generated by the following equations.
\begin{align}
	\frac{dp}{dt}=v\ \ \ \ \ \ \ \ \ \  \frac{Dv}{dt}=\yv
	\label{d2tg}
\end{align}
where  $p\in M$ and $\frac{D}{dt}$ is the covariant derivative. This is a simple way of consider the Gaussian Thermostat, but for our purposes it is better to consider an equivalent definition.

Let us consider two 1-forms on $TM$ as follows
\begin{align*}
	\al\left(\left(v_1, v_2\right)\right)=g_{p}\left(v,v_1\right)\ \ \ \ \ \ \ \ \be\left(\left(v_1,v_2\right)\right)=g_{p}(E(p),v_1)
\end{align*}
for every  $(v_1, v_2)\in T_{\te}TM\equiv T_{p}M\times T_{p}M$. For simplicity,  we denote $g_{p}(\cdot,\cdot)=\left\langle \cdot,\cdot\right\rangle$, sometimes we will also omit $p$ . Let  $\tilde{Y}:TM\re TM$ the operator such that for every  $u\in T_{p}M$ we have
\begin{align*}
	\tilde{Y}_{\te}(u)=\frac{1}{{\left|v\right|}^{2}}\left(\left\langle v,u\right\rangle E-\left\langle E,u\right\rangle v\right).
\end{align*}
We notice that for every $\te\in TM$,  $\tilde{Y}_{\te}:T_{p}M\re T_{p}M$ is anti-symmetric.

Let $\fg$ the 2-form on $TM$ such that for every $\te\in TM$ we have
\begin{align*}
	\fg_{\te}(\cdot,\cdot)=\left\langle \tilde{Y}_{\te}\left(d_{\te}\pi(\cdot)\right),d_{\te}\pi(\cdot)\right\rangle.
\end{align*}

Then, we define a vector field $F:TM\re TTM$  such that for $\xi\in T_{\te}TM$ we have
\begin{align*}
	\left\langle \left\langle d_{\te}H,\xi\right\rangle\right\rangle=d_{\te}H(\xi)=\om_{\te}\left(F(\te),\xi\right)
\end{align*}
Here $\left\langle \left\langle ,\right\rangle\right\rangle$ is the Sasaki metric and $\om=\om_{0}+\fg$, where $\om_{0}$ is the canonical symplectic form of $TM$. Usually, we denote $dH=i_{F}\om$. 

We remark that since $k$ is only a non-degenerated 2-form we will not use the notation $X_H$ for symplectic gradients.

In general, the Gaussian Thermostats are not Hamiltonian. However they preserve the energy levels. Let $\ms=H^{-1}(e)$ where $e\in \rea$. For every $\te\in \ms$ the orbit $\sigma_{\te}(t)\equiv\vp(\te)$ satisfies

\begin{align*}
\frac{d}{dt}H\left(\sigma_{\te}(t)\right)=d_{\sigma_{\te}(t)}H\left(F\left(\sigma_{\te}(t)\right)\right)=\om_{\sigma_{\te}(t)}\left(F\left(\sigma_{\te}(t)\right),F\left(\sigma_{\te}(t)\right)\right)=0,
\end{align*}
since $\om$ is a 2-form. Thus, $H\left(\sigma_{\te}(t)\right)=H(\te)=\ms$   for every $t\in\rea$. In particular, $\vp(\ms)\subset \ms$.

It is not difficult to see that we can define the Gaussian Thermostat of $E$ as the flow $\vp$ on $TM$ generated by the vector field $F$. 

\begin{lemma}
The vector field $F$ satisfies
\begin{align}
	F(\te)=\left(v,\yv\right).
	\label{d3tg}
\end{align}
\end{lemma}

\begin{proof}

Let $(\xi_{h},\xi_{v})=\xi\in T_{\te}\ms$, we have 
\begin{align*}
	d_{\te}H(\xi)=\left\langle v,\xi_{v}\right\rangle
\end{align*}

In the other hand,
\begin{align*}
	\om_0(F(\te),\xi)=\left\langle F(\te)_{h},\xi_v \right\rangle -   \left\langle F(\te)_{v},\xi_h\right\rangle,
\end{align*}
and
\begin{align*}
	\fg\left(F(\te),\xi\right)=\left\langle \tilde{Y}_{\te}\left(d_{\te}\pi(F(\te))\right),d_{\te}\pi(\xi)\right\rangle= \left\langle \tilde{Y}_{\te}\left(F(\te)_h\right),\xi_h\right\rangle.
\end{align*}
So
\begin{align*}
\om(F(\te),\xi)=\left(\left\langle F(\te)_{h},\xi_{v} \right\rangle  -    \left\langle \xi_{h}, F(\te)_{v}-\frac{1}{\left|v\right|^2} \left[ \left\langle v,F(\te)_{h}\right\rangle E -   \left\langle E,F(\te)_{v}\right\rangle    v     \right]\right\rangle                                \right).
\end{align*}
By definition, we have \begin{align*}
	d_{\te}H(\xi)=\om\left(F(\te),\xi\right)
\end{align*}
Thus
\begin{align*}
\left\langle v,\xi_{v}\right\rangle=\left(\left\langle F(\te)_{h},\xi_{v} \right\rangle  -    \left\langle \xi_{h}, F(\te)_{v}-\frac{1}{\left|v\right|^2} \left[ \left\langle v,F(\te)_{h}\right\rangle E -   \left\langle E,F(\te)_{h}\right\rangle    v     \right]\right\rangle                                \right).
\end{align*}
Henceforth,
\begin{align*}
	F(\te)_{h}=v\ \ \  \textrm{e}\ \ \ \ F(\te)_{v}= E - \frac{\left\langle E,v\right\rangle}{\left|v\right|^2}v,
\end{align*}
\end{proof}

As in \cite{WOJTKOWSKI}, if the vector field $E$ has a global potential $U$ then the Gaussian Thermostat preserves the measure $\mu=e^{-\frac{(n-1)U}{{\left|v\right|}^{2}}}dx$, where $dx$ denotes the standard volume element.

\section{Jacobi Fields and Pollicott's Formula}

In this section we collect two useful formulas, the first one is a Jacobi-type equation for  Gaussian Thermostats, the second one is due to Pollicott that describes how the entropy varies in an Anosov family.

Let $\vp:SM\re SM$  a Gaussian Thermostat flow as before. Let   $t\mapsto \left(\ga,\dg\right)$  an orbit of $\vp$. In particular,
 \begin{align}
\frac{D}{dt}\dg=\ty_{\ga}(\dg)=\yg,
\label{e1tg}	 
 \end{align}
where $D/dt$ denotes the covariant derivative and  $\dg=d\ga/dt$. 

Let $z:(-\ep,\ep)\re TM$ such that $z(0)=\te$ and $\dot{z}(0)=\xi$. We let a variation $f(s,t)=\pi\left(\vp\left(z(s)\right)\right)$. Let $\jx(t)=\partial f /\partial s (0,t)$, $\ga_{s}(t)=f(s,t)$, and $\ga=\ga_{0}$. We recall the following identity from Riemannian geometry,
\begin{align*}
	\frac{D}{dt}\frac{D}{dt}\pas+R\left(\pat,\pas\right)\pat=\frac{D}{ds}\frac{D}{dt}\pas.
\end{align*}
By (\ref{e1tg}), we obtain
\begin{align*}
	\ddj(t)+R(\dg,\jx)\dg=\frac{D}{ds}(\ty_{\ga_{s}}(\dg_{s}))=\frac{D}{ds}\left( E(\ga_{s})-\frac{\left\langle E(\ga_{s}),\dot{\gamma_{s}}\right\rangle}{{\left|\dot{\gamma_{s}}\right|}^{2}}\dot{\gamma}_{s}          \right).
\end{align*}
In $s=0$, we obtain
\begin{align}
	\ddj+R(\dg,\jx)\dg     -   \nabla_{\jx}E  +   \frac{1}{\left|{\dg}\right|^{2}} \left(  \left\langle \nabla_{\jx}E,\dg\right\rangle\dg    +    \left\langle E,\dej\right\rangle\dg     -2 \left\langle \dej,\dg\right\rangle \frac{\left\langle E,\dg\right\rangle}{\left|\dg\right|^{2}}\dg      + \left\langle E,\dg\right\rangle \dej \right)=0,
	\label{cjtg}
 \end{align}
This is the Jacobi equation for Gaussian Thermostats.


Let $\ld\fl$ with $\la\in(-\eps,\eps)$ be a $C^{\infty}$ family of Anosov flows on a closed manifold $X$. By the structural stability theorem there exists maps $\all\in C^{s}(X)$, $\mt^{\la}\in C^{s}(X,X)$ (where $s$ depends on the unstable and stable foliations) such that
\begin{enumerate}
	\item[i)]$\al^{0}\equiv1$; $\mt^{0}\equiv I_{X}$,
	\item[ii)]$\mt^{\la}$ sends orbits of  $\vp^{0}$ on orbits of $\fl$,
	\item[iii)]$\all$ is a change of speed of $\vp^{0}$ that turns $\mt^{\la}$ a conjugacy. Moreover the maps $\ld \all, \mt^{\la}$ are $C^{\infty}$.
\end{enumerate}
We consider the Taylor expansion of  $\all$:
\begin{align*}
	\ld\all = 1  +\la\left(D_{0}\all\right)   +   \left(\al^{2}/2\right)\left(D_{0}^{2}\all\right)+  \dots
\end{align*}

Let $h(\la)$ the topological entropy of  $\fl$. We recall that the variance of $\varphi^ 0_t$ is defined as follows. If $m$ denotes the maximal entropy measure of $\varphi^0_t$ and $F:M\to \R$ is a H\"older continuous function then
$$Var(F)=\int_{-\infty}^{\infty}(\int F\circ \varphi_t^0.Fdm-(\int Fdm)^2)dt.$$

 In  \cite{Pollicott}, Pollicott obtains the following results.
\begin{theorem}
The first derivative of $h(\la)$ at $\la=0$ satisfies
\begin{align*}
	h'(0)  =h(0) \int_{X}\left(D_{0}\all\right)dm.
\end{align*}
Moreover, the second derivative of  $h(\la)$ at $\la=0$ satisfies
\begin{align}
	h''(0)= h(0)\left\{ Var\left(D_{0}\all\right) + \int_{X}\left(D_{0}^{2}\all\right)dm  + 2 \left(\int_{X}D_{0}\all dm\right)^{2}     - 2\int_{X}\left(D_{0}\all\right)^{2}dm  \right\},
\end{align}
where $m$ is the maximal entropy measure of  $\vp^{0}$ and Var is the variance of $\vp^{0}$.
\label{teo3.1tg}
\end{theorem}

\section{Some Reductions}

In this section, we will apply the results in the previous sections to reduce the proof of the theorem.

Let $\mt^{\lambda}:SM\re SM$ the map of the previous section which sends orbits of $\vp=\vp^{0}$ on orbits of $\fl$. Let $\te\in SM$ such that $\pt^0_{t}(\te)$ is a periodic orbit of period $T$. Let $\ga_{0}=\pi\left(\vp^0(\te)\right)$. Hence,
\begin{align*}
	\ga_{\la}(t):=\pi\left(\fl(\Theta^{\lambda}(\te))\right),
\end{align*}
gives a  $C^{\infty}$ variation of the curve $\ga_{0}$,  since $\ld\mt^{\la}$ is $C^{\infty}$. 

Moreover the curves $\ga_{\la}$ of the variation are closed with period $T_{\la}$. We will see how $T_{\la}$ varies with respect to  $\la$.

Given a curve $c:[0,a]\re M$, we denote its length by $L(c)$. Its energy is denoted by $\ene(c)$ as follows
\begin{align*}
	\ene(c)=\int_{0}^{a}\left\langle \dot{c}(t),\dot{c}(t)\right\rangle  dt.
\end{align*}

Let $\Psi_{t}^{\la}$  the reparametrization $\fl$ by the change of speed $\al^{\la}$. Hence, $\mt^{\la}$ is a conjugacy between $\Psi_{t}^{\la}$ and $\fl$. Thus the orbits of $\Psi_{t}^{\la}(\te)$ are closed with periods $T_{\la}$. In particular, $T$ and  $T_{\la}$ satisfy the following relation.
\begin{align}
	T_{\la}=\int_{0}^{T}\frac{1}{\al^{\la}\left(\gz(t)\right)}dt.
	\label{rtlt}
\end{align}
 Let $\bl(t):=\frac{1}{\al^{\la}\left(\gz(t)\right)}$. We obtain the following result. 
\begin{lemma}
For every closed geodesic $\ga_{0}$ with period $T$ we have:
\begin{align*}
	\izt D_{0}\be\left(\gz(t)\right)dt=0\ \ \ \ \ \ \ \ and \ \ \ \ \ \ \ \izt D_{0}^{2}\be\left(\gz(t)\right)dt=\frac{1}{2}\frac{d^{2}\ene}{d\la^{2}}\Big|_{\la=0}. 
\end{align*}
\label{lema4.1tg}
\end{lemma}
\begin{proof} Since $\left|\gl\right|=1$ we have
\begin{align}
	L(\gl)=\tl=\izt\bl\left(\gz(t)\right)dt.
	\label{e6ptg}
\end{align}
 
Now, we consider a new family of curves $\glb:[0,T]\re M$ given by
\begin{align*}
	\glb(t):=\ga_{\la}(t\tl/T).
\end{align*}
The map $\ld\glb $ gives rise a $C^{\infty}$ variation of the closed geodesic  $\ga_{0}$ by closed curves with period $T$ (however, not necessarily we have $\left|\glbp\right|=1$). Let us denote $\ene_{\la}=\ene(\glb)$. We obtain,
\begin{align*}
	\ene_{\la}=\frac{\tl^{2}}{T}.
\end{align*}
 Differentiating with respect to  $\la$ we obtain,
\begin{align}
	\frac{d\ene_{\la}}{d\la}=\frac{2\tl}{T}\frac{d\tl}{d\la}.
	\label{dpen}
\end{align}
Taking $\la=0$ in (\ref{dpen}), we obtain, from (\ref{e6ptg}), that
\begin{align}
	\izt D_{0}\be\left(\gz(t)\right)dt    = \frac{d\tl}{d\la}\Big|_{\la=0}   =\frac{1}{2} \frac{d\ene}{d\la}\Big|_{\la=0}.
	\label{e93tg}
\end{align}
On the other hand, $\ga_{0}$ is a closed geodesic. Thus it is a critical point of the energy. So, 
\begin{align}
	\frac{d\ene}{d\la}\Big|_{\la=0}=0.
	\label{e97tg}
\end{align}
This, joint with equation (\ref{e93tg}) gives the proof of the first part of the Lemma. 

Moreover, differentiating (\ref{dpen}) with respect to $\la$, evaluating at $\la=0$ and using  (\ref{e6ptg}) we obtain
\begin{align*}
	\izt D_{0}^{2}\be\left(\gz(t)\right)dt    = \frac{d^{2}\tl}{d\la^{2}}\Big|_{\la=0}   =\frac{1}{2} \frac{d^{2}\ene}{d\la^{2}}\Big|_{\la=0}   -  \frac{1}{T}\left(\frac{d\tl}{d\la}\right)^{2}\Big|_{\la=0}.
\end{align*}
Using the first part of the lemma, this completes the proof.
\end{proof}

So, we conclude that
\begin{align}
	\int_{SM}D_{0}\be dm=0 \ \ \ \ \ \ e \ \ \ \ \ \ Var\left(D_{0}\be\right)\equiv0,\ \ \ \  \ \ \ \ \ \ \ \ \ \ \ \ \
\label{evtg}
\end{align}
where $Var$ denotes the variance (as in \cite{Pollicott}). 

Indeed, since $\int_{0}^{T}D_{0}\be\left(\gz(t)\right) dt= 0$, by Livisc's Theorem \cite{Livsic}, we have that the map  $\te\re D_{0}\be(\te)$ is zero up to a co-boundary. Using the properties of the variance, we obtain that $Var\left(D_{0}\be\right)\equiv 0$.

Moreover, since $\bl=\frac{1}{\al^{\la}}$, we have,
\begin{align*}
	D_{\la}^{2}\be=\frac{2}{\left(\al^{\la}\right)^{3}}\left(D_{\la}\al^{\la}\right)^{2}    -    \frac{1}{\left(\al^{\la}\right)^{2}}D_{\la}^{2}\al^{\la}.
\end{align*}
Thus, at $\la=0$, we have,
\begin{align}
	D_{0}^{2}\be=2\left(D_{0}\al\right)^{2}     -D_{0}^{2}\al,
	\label{ed2tg}
\end{align}
once that $\al^{0}\equiv 1$. 

Using equations (\ref{evtg}) e (\ref{ed2tg}) of Theorem\ \ref{teo3.1tg}, we obtain,
\begin{align}
	h'(0)=0, \textrm{ and }\ \ \ \ \ \ \ \ \ \ \ \ \  \ \ \ \ \ \ \ \ \ \ \  
	\label{edhtg}
\end{align}
\begin{align}
	h''(0)=-h(0)    \int_{SM}D_{0}^{2}\be dm.\ \ \ \ \ \  
	\label{ed2htg}
\end{align}
Hence, to conclude the proof of Theorem \ref{teo3.1tg} it is enough to estimate $\int_{SM}D_{0}^{2}\be dm$.

\section{The Variational Field}

In this section, we obtain an useful equation for the variational field.

\begin{lemma}
Let $W=\frac{\partial \glb }{\partial \la}\big|_{\la=0}$ be the variational field associated to the variation $\ld \glb$ of the closed geodesic $\ga_{0}$. Then:
\begin{align*}
	W(0)=W(T)\ \ \ \ \ \ \ \ \ \ \ \ \ \ \ \ \ \ \ \ \ \ \ \ \ \ \ \\\
	\dot{W}(0)=\dot{W}(T)\ \ \ \ \ \ \ \ \ \ \ \ \ \ \ \ \ \ \ \ \ \ \ \ \ \ \ \\\
	\ddot{W}+R(\gz(t),W(t))\gz(t)=\frac{dE_{\la}}{d\la}\Big|_{\la=0}   - \frac{1}{\left|\gz\right|^{2}}\left\langle \frac{dE_{\la}}{d\la}\Big|_{\la=0} , \gz    \right\rangle \gz.
	\end{align*}
\label{lema4.3}
\end{lemma}

\begin{proof} By definition of $\glb$, we have,
\begin{align*}
	\glb(t)=\gamma_{\la}(tT_{\la}/T)
\end{align*}
where $\glb(t)=\pi\left(\fl(\Theta^{\la}(\te))\right)$.

Hence, $\glb(0)=\glb(T)$ and $\glbp(0)=\glbp(T)$. From the definition of $W(t)$, we obtain $W(0)=W(T)$. The equality $\dot{W}(0)=\dot{W}(T)$ follows from the symmetry of the Riemannian connection. 

Let us consider the surface given by  $(t,\la)\re f(t,\la)$. We recall the following identity from Riemannian geometry
\begin{align*}
	\frac{D}{dt}\frac{D}{dt}\frac{\partial f}{\partial \la}    + R\left(\frac{\partial f}{\partial t} ,\frac{\partial f}{\partial \la} \right)\frac{\partial f}{\partial t}   =    \frac{D}{d\la}\left(\frac{D}{dt}\pat\right).
\end{align*}
Setting $f(t,\la)=\glb$, by the identity above, at $\la=0$ we have,
	\begin{align*}
		\ddot{W}  + R(\gz,W)\gz=\frac{D}{d\la}\left(\frac{D}{dt}\glbp\right)\Big|_{\la=0}.
	\end{align*}
Now, by the definition of the Gaussian Thermostat, we have,
	\begin{align*}
		\frac{D}{dt}\gl=\yygl,
	\end{align*}
Thus,
	\begin{align*}
		\frac{D}{dt}\glbp(t)=\frac{D}{dt}\left[\gl(\frac{tT_{\lambda}}{T})\right]=    \frac{T_{\la}}{T}\left(E_{\lambda}\left(\glbp(t)\right)      -\frac{\left\langle E_{\lambda}    \left(\glbp(t)\right), \glbp(t)       \right\rangle}{{\left|\glbp(t)\right|}^{2}}\glbp(t)\right).
	\end{align*}
Hence
\begin{align}
	\dcl\left(\dct\glbp\right)\Bigg|_{\la=0}=\dcl\left[ \frac{T_{\la}}{T}\left(E_{\lambda}       \left(\glbp(t)\right)      -\frac{\left\langle E_{\lambda}    \left(\glbp(t)\right), \glbp(t)       \right\rangle}{{\left|\glbp(t)\right|}^{2}}\glbp(t)\right)          \right]\Bigg|_{\la=0}\nonumber\\
	=      \dcl\left[\left( E_{\lambda}\left(\glbp(t)\right)      -\frac{\left\langle E_{\lambda}    \left(\glbp(t)\right), \glbp(t)       \right\rangle}{{\left|\glbp(t)\right|}^{2}}\glbp(t)\right)          \right]\Bigg|_{\la=0}.\ \ \  \ 
	\label{e60tg}
\end{align}
Indeed, Lemma\ \ref{lema4.1tg},\ says that  $\frac{d}{dt}(T_{\la}/T)|_{\la=0}=dE/dt|_{\la=0}=0$ and $T_{0}=T$. 

Now, we compute each term in  (\ref{e60tg}).

We have,
\begin{align*}
	\dcl E_{\la}\left(\gl(t)\right)\Big|_{\la=0}=\frac{dE}{d\la}\Big|_{\la=0}+  \nabla_{W}E_{0}\\
	= \frac{dE}{d\la}\Big|_{\la=0}\ \ \ \ \ \ \ \ \ \ \  \ 
\end{align*}
Since,  by hypothesis, $E_{0}\equiv0$. 

We also have,
\begin{align*}
	\dcl\left(  \frac{\left\langle E_{\lambda}    \left(\glbp(t)\right), \glbp(t)       \right\rangle}{{\left|\glbp(t)\right|}^{2}}    \glbp(t)          \right) \Bigg|_{\la=0}   =      \dcl\left(  \frac{\left\langle E_{\lambda}    \left(\glbp(t)\right), \glbp(t)       \right\rangle}{{\left|\glbp(t)\right|}^{2}}          \right) \Bigg|_{\la=0}\gz         \\
				+ \frac{\left\langle E_{0},\gz(0)\right\rangle}{\left|\gz(0)\right|^{2}}\dcl\gl\big|_{\la=0}\ \ \ \ \ \ \ \ \ \ \ \ \ \ \ \ \\
\ 	=      \dcl\left(  \frac{\left\langle E_{\lambda}    \left(\glbp(t)\right), \glbp(t)       \right\rangle}{{\left|\glbp(t)\right|}^{2}}          \right) \Bigg|_{\la=0}\gz,
\end{align*}
and
\begin{align*}
\dcl\left(  \frac{\left\langle E_{\lambda}    \left(\glbp(t)\right), \glbp(t)       \right\rangle}{{\left|\glbp(t)\right|}^{2}}          \right) \Bigg|_{\la=0}   =    \frac{1}{\left|\gz\right|^{2}}    \left(  \frac{d}{d\la} \left(\left\langle E_{\la}, \glbp(t) \right\rangle\right)   \Big|_{\la=0}          \right) \ \\
\ -      \left(        \left\langle E_{0},\gz\right\rangle   2 \left\langle \dcl \glbp , \glbp \right\rangle                                     \right)\frac{1}{\left|\gz\right|^{4}}   \\
 = \frac{1}{\left|\gz\right|^{2}}\left\langle \frac{dE_{\la}}{d\la}\big|_{\la=0} , \gz    \right\rangle \gz   .  \ \ \ \ \ \ \ \  \ 
\end{align*}
\noi
Thus, we conclude that 
\begin{align*}
	\dcl\left(\dct\glbp\right)\Bigg|_{\la=0}	=  \frac{dE_{\la}}{d\la}\Big|_{\la=0}   - \frac{1}{\left|\gz\right|^{2}}\left\langle \frac{dE_{\la}}{d\la}\Big|_{\la=0} , \gz    \right\rangle \gz.
\end{align*}
This completes the proof.
\end{proof}

\section{Proof of Theorem \ref{teo1.1tg}}

In this section we prove Theorem \ref{teo1.1tg}.

Let $\Lambda$ the vector space of piecewise differentiable vector fields $V:[0,T]\re \ml$ along the closed geodesic $\ga_{0}:[0,T]\re M$ such that $V(0)=V(T)$. 

We also consider the index form of $\ga_{0}$, $I:\ml\times\ml\re\rea$ given by
\begin{align}
	I(U,V)=\int_{0}^{T}  \left\langle \dot{U},\dot{V}\right\rangle   -   \left\langle R(\gz,U)\gz,V\right\rangle dt.
	\label{fitg}
\end{align}
Since the vector fields of  $\ml$ satisfy $V(0)=V(T)$, we obtain, 
\begin{align}
	\frac{1}{2}\frac{d^{2}{\mathcal{E}}}{d\la}\Big|_{\la=0}=I(W,W),
	\label{e69tg}
\end{align}
where as in Lemma \ref{lema4.3}, $W$ is the variational field of the variation  $\ld \bar{\gamma}_{\lambda}$. 

Since $\ptt^0$ is Anosov, then the metric has no conjugated points. So, by Morse's Index Theorem, we obtain,
\begin{align}
	I(V,V)\geq 0.
	\label{e70tg}
\end{align}
 for every $V\in \ml$.  By Lemma \ref{lema4.3}, we have that $W(t)\in \ml$. 
 
 Obviously,  $E'_0\left(t\right)-\frac{\left\langle E'_{0}(t),\dot{\gamma_{0}}(t)\right\rangle}{{\left|\dot{\gamma_{0}}(t)\right|}^{2}}\dot{\gamma}_{0}(t)\in \ml$, where $E'_0(t)=E'\left(\ga_0(t)\right)$. 
 
 In the following, we will omit $t$.  We have that  $\left[W+x\left(E'_0-\ygpz\right)\right]\in\ml$ for every $x\in\rea$. Thus, using (\ref{e70tg}) we have,
 \begin{align*}
	 I\left(W+x\left(E'_0-\ygpz\right),W+x\left(E'_0-\ygpz\right)\right)  \geq 0.
 \end{align*}
Therefore,
\begin{align}
	I(W,W)\geq -2x I\left(W,E'_0-\ygpz\right)   -x^{2}I\left(E'_0-\ygpz,E'_0-\ygpz\right).
	\label{e73tg}
\end{align}

Now, we compute each term of  (\ref{e73tg}). We have, 

\begin{align}
	I\left({\pe}-\frac{\left\langle {\pe},\dot{\gamma_{0}}\right\rangle}{{\left|\dot{\gamma_{0}}\right|}^{2}}\dot{\gamma_{0}},   {\pe}-\frac{\left\langle {\pe},\dot{\gamma_{0}}\right\rangle}{{\left|\dot{\gamma_{0}}\right|}^{2}}\dot{\gamma_{0}}       \right) = I(\pe,\pe)  \ \ \ \ \ \ \ \ \ \ \ \ \ \ \ \ \ \nonumber\\
	   - 2I\left(\pe,\ygpz\right)  + I\left(\ygpz,\ygpz\right).
		\label{e001tg}
\end{align}
By the definition of the index form of $\ga_{0}$, given by (\ref{fitg}),  we have
 \begin{align}
	 I\left(\pe,\pe\right)= \int_{0}^{T}\left\langle \nabla_{\gz} \pe, \nabla_{\gz} \pe  \right\rangle     -      \left\langle R(\gz,\pe)\gz,\pe\right\rangle dt\ \ \ \nonumber\\
	= \int_{0}^{T}{\left|\nabla_{\gz} \pe \right|}^{2}     -      \left\langle R(\gz,\pe)\gz,\pe\right\rangle dt.\  \ \ \ \ \ \ \ \ \ \ \ \ \ \ 
	\label{e0002tg}
 \end{align}
We also have,
\begin{align}
	I\left(\pe,\ygpz\right)=   \int_{0}^{T}\left\langle \nabla_{\gz} \pe,\frac{D}{dt}\left(\ygpz\right)  \right\rangle \ \ \ \ \nonumber\\
			-      \left\langle R(\gz,\pe)\gz,\ygpz\right\rangle dt\nonumber \ \ \ \ \ \ \ \ \ \\
	 =\int_{0}^{T}\left\langle \nabla_{\gz} \pe,\frac{D}{dt}\left(\ygpz\right)  \right\rangle dt.  
	\label{e002tg}
\end{align}
However,
\begin{align*}
	\frac{D}{dt}\left(\ygpz\right)=\frac{d}{dt}\left(\frac{\left\langle \pe,\gz\right\rangle}{\left|\gz\right|^{2}}\right)\gz  + \frac{\left\langle \pe,\gz\right\rangle}{\left|\gz\right|^{2}}\frac{D}{dt}\gz  \ \ \ \ \ \ \ \ \ \  \ \ \\
	= \frgz\left\langle \nabla_{\gz}\pe,\gz\right\rangle\gz.   \ \ \ \ \ \ \ \ \ \ \ \ \ \ \ \ \ \ \  \ \ \ \ \ \ \ \ \ \ \ \ \ \
\end{align*}
Thus, using (\ref{e002tg}), we obtain
\begin{align}
	I\left(\pe,\ygpz\right)=   \int_{0}^{T}\frac{\left\langle \nabla_{\gz} \pe,\gz \right\rangle ^{2}}{\left|\gz\right|^{2}} dt.
	\label{e005tg}
\end{align}
Moreover,
\begin{align}
	I\left(\ygpz,\ygpz\right)=\int_{0}^{T}\left\langle \frac{D}{dt}\left(\ygpz\right),\frac{D}{dt}\left(\ygpz\right)\right\rangle       \nonumber \\
	- \left\langle R\left(\gz,\ygpz\right)\gz,\ygpz\right\rangle dt\ \ \ \ \  \nonumber\\
	= \int_{0}^{T}\frac{\left\langle \nabla_{\gz} \pe,\gz \right\rangle ^{2}}{\left|\gz\right|^{2}} dt.    \ \ \ \ \ \ \ \ \ \ \ \  \ \ \ \ \ \ \ \ \ \ \ \ \ \ \ \ \ \  
	\label{e007tg}
\end{align}

Replacing this terms,  (\ref{e0002tg}), (\ref{e005tg}) e (\ref{e007tg}), in  (\ref{e001tg}), we obtain
\begin{align}
		I\left(\pe-\frac{\left\langle \pe,\dot{\gamma_{0}}\right\rangle}{{\left|\dot{\gamma_{0}}\right|}^{2}}\dot{\gamma_{0}},   \pe      -   \frac{\left\langle \pe,\dot{\gamma_{0}}\right\rangle}{{\left|\dot{\gamma_{0}}\right|}^{2}}\dot{\gamma_{0}}       \right)     =   \int_{0}^{T}{\left|\nabla_{\gz} \pe \right|}^{2}     -\frac{\left\langle \nabla_{\gz} \pe,\gz \right\rangle ^{2}}{\left|\gz\right|^{2}}       \nonumber\\
		-	\left\langle R(\gz,\pe)\gz,\pe\right\rangle dt.\ \ \ \ \ \ \ \ \ \ \ \
		\label{Btg}
\end{align}

We also have,
\begin{align}
	I\left(W,\pe-\ygpz\right)=I\left(W,\pe\right)    -  I\left(W,\ygpz\right).
	\label{009tg}
\end{align}

But,
\begin{align*}
	I\left(W,\pe\right)=\int_{0}^{T}\left\langle \ddot{W},\pe\right\rangle     -   \left\langle R(\gz,W)\gz,\pe\right\rangle dt\ \ \nonumber\\
	=\int_{0}^{T}-{\left|\pe\right|}^{2}   +   \frac{\left\langle \nabla_{\gz} \pe,\gz \right\rangle ^{2}}{\left|\gz\right|^{2}}dt.\ \ \ \ \ \ \ \ \ \ \ 
\end{align*}

Lemma \ref{lema4.3} says that
\begin{align*}
	I\left(W,\ygpz\right)=\int_{0}^{T}\left\langle \ddot{W},\ygpz\right\rangle     -   \left\langle R(\gz,W)\gz,\ygpz\right\rangle dt\     \ \ \\
	= \int_{0}^{T}   \left\langle R(\gz,W)\gz    \pe +   \ygpz,\pe\right\rangle dt \ \ \ \ \ \ \ \ \ \ \ \ \ \ \ \   \\
	=0.\    \ \ \ \ \ \ \ \ \ \ \ \ \ \ \ \ \ \ \ \ \ \ \ \ \ \ \ \ \ \ \ \ \ \ \ \ \ \ \ \ \ \ \ \ \ \ \ \ \ \ \ \ \ \ \ \ \ \ \ \ \ \ \ \ \ \ \ \ \ \ \ \ \ \ \ \ \ 
\end{align*}

Therefore,
\begin{align}
	I\left(W,\pe- \ygpz \right)	=  \int_{0}^{T}-\left|\pe\right|^{2}   +   \frac{\left\langle \nabla_{\gz} \pe,\gz \right\rangle ^{2}}{\left|\gz\right|^{2}}dt.
	\label{Atg}
\end{align}

In the other hand, using Lemma\ \ref{lema4.1tg} and equation (\ref{e69tg}), we obtain
\begin{align*}
	\int_{0}^{T} {D}_{0}^{2}\be \left(\dot{\gamma}_{0}\right) dt = \frac{1}{2} \frac{{d}^{2}E}{d\la^{2}} \big|_{\lambda=0}=I(W,W).
\end{align*}

This give us,
\begin{align}
	\int_{0}^{T}D_{0}^{2}\be\left(\gz\right) dt 	\geq       2x \int_{0}^{T}   \left|\pe\right|^{2}   -   \frac{\left\langle \nabla_{\gz} \pe,\gz \right\rangle ^{2}}{\left|\gz\right|^{2}}dt \ \ \ \ \ \ \ \ \ \ \ \ \ \ \ \ \ \ \ \ \ \ \ \ \ \ \ \ \  \nonumber\\
					- x^{2}\int_{0}^{T} {\left|\nabla_{\gz} \pe \right|}^{2}     - \frac{\left\langle \nabla_{\gz} \pe,\gz \right\rangle ^{2}}{\left|\gz\right|^{2}}     -      \left\langle R(\gz,\pe)\gz,\pe\right\rangle dt.
	\label{e75tg}
\end{align}

Now, since $\vp=\vp^0$ is a volume-preserving Anosov flow, we have that the set of probability measures supported on periodic orbits is dense in the set of invariant probability measures. Hence, by (\ref{e75tg}) we obtain
\begin{align}
	\int_{SM}D_{0}^{2}\be\left(\gz\right) dm \geq   				2x \int_{SM}\left|\pe\right|^{2}   -   \frac{\left\langle \nabla_{\gz} \pe,\gz \right\rangle ^{2}}{\left|\gz\right|^{2}}dm \ \ \ \ \ \ \ \ \ \ \ \ \ \ \ \ \ \ \ \ \ \ \ \ \ \ \ \ \ \ \ \ \ \ \ \ \ \ \  \nonumber\\       
	- x^{2}\int_{SM}{\left|\nabla_{\gz} \pe \right|}^{2}     -\frac{\left\langle \nabla_{\gz} \pe,\gz \right\rangle ^{2}}{\left|\gz\right|^{2}}     -      \left\langle R(\gz,\pe)\gz,\pe\right\rangle dm:= 2x A -x^{2}B,\ \ \ \ 
	\label{e75.1tg}
\end{align}

for every $x\in\rea$. Since, $B=I(V,V)$, with $V=E'_0-\ygpz\in \Lambda$, we obtain from (\ref{e70tg}) that $B\geq0$. 

From the above equation, if  $Z\neq0$ (which is equivalent to $A>0$) then $B>0$. 

Now, we consider, $f:\rea\re\rea$ given by $f(x)=2x A-x^{2}B$. Then,
\begin{align*}
	f'(x)=2A-2x B\ \  \ \ \ \ \textrm{e}\ \ \ \ \ f''(x)=-2B. 
\end{align*}

We have that  $0=f'\left(\frac{A}{B}\right)$ and $f''\left(\frac{A}{B}\right)=-2B<0$ . Indeed, by hypothesis $A>0$, hence, $B>0$.  This implies that  $f$ has a maximum at $\frac{A}{B}$, and its value is $f\left(\frac{A}{B}\right)=\frac{A^{2}}{B}>0$. 

Hence, by (\ref{ed2htg}) we have
\begin{align*}
	h''(0)=-h(0)  \int_{SM} D_{0}^{2}\be dm      \leq -h(0)(2x A- x^{2}B)    =   -h(0) \frac{A^{2}}{B}  < 0,
\end{align*}
where in the last equality we take $x=A/B$. 

In particular,
\begin{align*}
	h''(0)\leq -h(0) \frac{\left[\int_{SM}\left( \left|\pe\right|^{2}  - \frac{\left\langle \pe,v\right\rangle^{2}}{\left|v\right|^{2}}                 \right)      dm            \right]^{2}}{  \int_{SM} \left( \left|\nabla_{v}\pe\right|^{2} -\frac{\left\langle \nabla_{v}\pe,v\right\rangle^{2}}{\left|v\right|^{2}}  - \left\langle R(v,\pe)v,\pe\right\rangle                     \right)dm                 } < 0.
\end{align*}

And this completes the proof.


\begin{thebibliography}{00}





\bibitem{Burns} Burns, K; Paternain, G.P. {\it Anosov magnetic flow, critical values and topological entropy}, Nonlinearity {\bf 15}, 281-314, 2002.
\bibitem{Galavotti} Gallavotti, G. {\em New methods in nonequilibrium gases and fluids},Open Sys. Inf. Dynam. 6, 101-136. (1999)
\bibitem{PP} Paternain, G. P; Paternain, M. {\it Topological entropy and magnetic fields}. Chaos in Uruguay , 47-53, Univ. Repub., Montevideo, 1997.
\bibitem{P.deriv.geo} Paternain, G. P; Paternain, M. {\it First derivative of topological entropy for Anosov geodesic flows in the presence of magnetic fields}, Nonlinearity {\bf 10}, 121-131, 1997.
\bibitem{P.livro} Paternain, G.P. {\it Geodesic flow}, Progress in Mathematics {\bf 180}, Birkhaüser 1999.
\bibitem{Livsic} Livsic, A. {\it Cohomology of dinamical systems}, Math. USSR-Izv, {\bf 6}, 1278-1301, 1972.
\bibitem{Pollicott} Pollicott, M. {\it Derivatives of topological entropy for Anosov and geodesic flow}, J. Diff. Geom. {\bf 39}, 457-489, 1994.
\bibitem{Ruelle}
Ruelle, D.
{\em Smooth dynamics and new theoretical ideas in nonequilibrium statistical mechanics},
J. Stat.
Phys. 95, 393?468 (1999).
\bibitem{Walters} Walters, P. {\it An introduction to ergodic theory}, Graduate Text in Mathematics 79, Spriger, Berlin, 1982.
\bibitem{WOJTKOWSKI} Wojtkowski, M. {\it Magnetic flows and gaussian thermostats on manifolds of negative curvature}, Fundamenta Mathematicae, vol. 163,  177-191, 2000.

\end{thebibliography}
\end{document}